\documentclass[11pt]{article}
\usepackage[margin=1in]{geometry}
\usepackage{tikz}
\usepackage[small,bf,hang]{caption}
\usepackage{float}
\usepackage{siunitx}
\usepackage{amsmath,amsthm,amssymb}
\usepackage{graphicx}
\usepackage{subfig}
\usepackage{hyperref}
\hypersetup{colorlinks=true,    citecolor = blue,    linkcolor=black,    filecolor=magenta,    urlcolor=cyan}
\newtheorem{theorem}{Theorem}

\newtheorem{lemma}{Lemma}
\newtheorem{claim}{Claim}

\newcommand{\sm}{\setminus}

\usepackage[utf8]{inputenc}  
\usepackage[T1]{fontenc} 
\usepackage{amsthm}
\usepackage{amsmath}
\usepackage{amssymb}
\usepackage{enumitem}
\usepackage{xcolor}
\usepackage{amsfonts}
\usepackage{fdsymbol}
\usepackage{hyperref}
\theoremstyle{remark}

\newcommand{\cC}{\mathcal{C}}
\newcommand{\cP}{\mathcal{P}}
\newcommand{\cR}{\mathcal{R}}

\title{A positive instance of Scott's Conjecture on induced subdivisions}
\author{Kathie Cameron
 \thanks{Department of Mathematics, Wilfrid Laurier University,
 Waterloo, ON, Canada, N2L 3C5. Email: \texttt{kcameron@wlu.ca}. ORCID: 0000-0002-0112-2494. 
 We acknowledge the support of the Natural Sciences and Engineering Research Council of Canada (NSERC), [funding reference number RGPIN-2016-06517]. Cette recherche a \'et\'e financ\'ee par le Conseil de recherches en sciences naturelles et en g\'enie du Canada (CRSNG), [num\'ero de r\'ef\'erence RGPIN-2016-06517].}
 \and Ni Luh Dewi Sintiari
 \thanks{Department of Informatics, Universitas Pendidikan Ganesha,
 Bali, Indonesia, 81116. Email: \texttt{luh.dewi.sintiari@undiksha.ac.id}. ORCID: 0000-0002-6562-4434. We acknowledge the support of the Natural Sciences and Engineering Research Council of Canada (NSERC), [funding reference number RGPIN-2016-06517]. Cette recherche a \'et\'e financ\'ee par le Conseil de recherches en sciences naturelles et en g\'enie du Canada (CRSNG), [num\'ero de r\'ef\'erence RGPIN-2016-06517].}
 \and Sophie Spirkl\thanks{Department of Combinatorics and Optimization, University of Waterloo, Waterloo, ON, Canada, N2L 3G1. We acknowledge the support of the Natural Sciences and Engineering Research Council of Canada (NSERC), [funding reference number RGPIN-2020-03912]. Cette recherche a \'et\'e financ\'ee par le Conseil de recherches en sciences naturelles et en g\'enie du Canada (CRSNG), [num\'ero de r\'ef\'erence RGPIN-2020-03912]. This project was funded in part by the Government of Ontario. This research was conducted while Spirkl was an Alfred P. Sloan Fellow.}}

\begin{document}
\maketitle
\begin{abstract}
For a graph $G$, $\chi(G)$  denotes the chromatic number of $G$ and $\omega(G)$ denotes the size of the largest clique in $G$.  A hereditary class of graphs is called $\chi$-bounded if there is a function $f$ such that for each graph $G$ in the class, $\chi(G) \le f(\omega(G))$.

Scott (1997) conjectured that for every graph $H$, the class of graphs which do not contain any subdivision of $H$ as an induced subgraph is $\chi$-bounded. He proved his conjecture when $H$ is a tree and when $H$ is the complete graph on four vertices, $K_4$. Esperet and Trotignon (2019) proved that the conjecture holds when $H$ is $K_4$ with one edge subdivided once. 

Scott's conjecture was disproved by Pawlik et al.\ (2014). Chalopin et al.\ (2016) gave more counterexamples including the graph obtained from $K_4$ by subdividing each edge of a 4-cycle once.  

We prove that the conjecture holds when $H$ consists of a complete bipartite graph with and additional vertex which has exactly two neighbours, on the same side of the bipartition. As a special case, this proves Scott's conjecture when $H$ is obtained from $K_4$ by subdividing two disjoint edges.  
\end{abstract} 

\section{Introduction}

All graphs in this paper are finite and simple. 
A \emph{clique} in a graph is a set of pairwise adjacent  vertices.  
An \emph{independent set} in a graph is a set of pairwise non-adjacent vertices. A  \emph{colouring} of a graph is a partition of its vertices into independent sets; the \emph{chromatic number} of the graph is the minimum number of independent sets required in a colouring. 
For a graph $G$, let $\chi(G)$  denote the chromatic number of $G$ and let $\omega(G)$ denote the size of the largest clique in $G$.  
A class of graphs is called \emph{hereditary} if it is closed under taking induced subgraphs and under isomorphism. 

A hereditary class of graphs is called $\chi$\emph{-bounded} if there is a fixed function $f$ such that for each graph $G$ in the class, $\chi(G) \le f(\omega(G))$  \cite{GyarfasChiBdd}; $f$ is called a \emph{binding function}. For example, perfect graphs are the unique maximal class of graphs which are $\chi$-bounded with binding function the identity.

For graphs $G$ and $H$ and a set $\mathcal{H}$ of graphs, $G$ is called \emph{$H$-free} if G has no induced subgraph isomorphic to $H$ and $G$ is $\mathcal{H}$\emph{-free} if $G$ has no induced subgraph isomorphic to any graph in $\mathcal{H}$. Subdividing an edge means replacing the edge by a path of length at least one. A subdivision of a graph is obtained by subdividing some of its edges. For graphs $G$ and $H$, $G$ is said to be IS$H$-free if $G$ has no induced subgraph isomorphic to a subdivision of $H$. 

Erd\H os  \cite{Erdos} showed that for positive integers $g$ and $k$, there are graphs with girth higher than $g$ and chromatic number at least $k$. Thus, if $H$ is a graph which contains a cycle, the class of $H$-free graphs is not $\chi$-bounded. Gy\'arf\'as \cite{GyarfasSC} and Sumner \cite{Sumner} conjectured the converse, that for every tree $T$, the class of $T$-free graphs is $\chi(G)$-bounded. Scott \cite{Scott} proved a topological version of the Gy\'arf\'as-Sumner Conjecture:  The class of IS$T$-free graphs is $\chi$-bounded.  

More generally, Scott \cite{Scott} conjectured that for any graph $H$, the class of IS$H$-free graphs is $\chi$-bounded. Scott also proved his conjecture when $H$ is the complete graph on four vertices, $K_4$ (see  \cite{LMT}). Scott's Conjecture  is known to hold for all graphs on at most four vertices and some other graphs (see, for example,  \cite{MCIPASNT}). Scott's Conjecture was disproved by Pawlik et al.\ \cite{ScottConjFalse}; using the constructino of \cite{ScottConjFalse}, Chalopin et al.\ \cite{Chalopin} showed that Scott's conjecture does not holds for the graph obtained from $K_4$ by subdividing each edge of a 4-cycle once. 
Finding graphs $H$ for which the conjecture holds remains intriguing. 

The complete bipartite graph $K_{m,n}$ consists of two independent sets of vertices, $U$ of size $m$ and $W$ of size $n$, and all possible edges with one end in $U$ and the other end in $W$. 
The graph $K_4^+$ is $K_4$ with one edge subdivided and $K_4^{++}$ is $K_4$ with two disjoint edges subdivided. Esperet and Trotignon \cite{K4+} proved that the class of IS$K_4^+$-free graphs is $\chi$-bounded. 

Our main result is that the conjecture holds for bipartite graphs $H$ of a particular form: a complete bipartite graph together with a vertex adjacent to exactly two vertices in one part of the bipartition. This implies that IS$K_4^{++}$-free graphs are $\chi$-bounded. This also  implies that when $H$ is $K_{m,n}$ together with an additional vertex of degree 1,  then IS$H$-free graphs are $\chi$-bounded.

\medskip

\section{Preliminaries}

Let $G$ be a graph with vertex-set $V(G)$ and edge-set $E(G)$. For a vertex $v$ of $G$, the \emph{neighbour-set} of $v$ is set of vertices adjacent to $v$, and is denoted by $N_G(v)$ or $N(v)$. Let $S$ be a subset of $V(G)$. The \emph{neighbour-set} of $S$, denoted $N_G(S)$ or $N(S)$, is the set of vertices outside of $S$ which have a neighbour in $S$. 

Let $a$ be a nonnegative integer. A vertex $v \in G \sm S$ is \emph{$a$-disconnected} to $S$ if $v$ has at most $a$ non-neighbours in $S$; when $a=0$, we simply say that $v$ is complete to $S$. Analogously,  $v \in G \sm S$ is  \emph{$a$-connected} to $S$ if $v$ has at most $a$ neighbours in $S$, and when $a=0$ we say that $v$ is anticomplete to $S$. 
The same terminology is used when vertex $v$ is replaced by a set of vertices. 
\medskip

We will use the following results:

\begin{theorem}[K{\"u}hn and Osthus,  \cite{KuhnOsthus}]
\label{thm:kuhn-osthus}
Let $H$ be a graph and let $s$ be a positive integer. 
Then there is an integer $d = d(H, s)$ such that every graph with  average degree at least $d$  contains either 
an induced subdivision of $H$ or a complete bipartite (not necessarily induced) subgraph $K_{s,s}$.

\end{theorem}	

Recently, two groups of authors independently proved a strengthening of the above, giving polynomial bounds (here we state the version due to \cite{polyko2}, which gives more explicit bounds): 
\begin{theorem}[Bourneuf, Buci\'c, Cook, and Davies \cite{polyko1}, Girão and Hunter \cite{polyko2}] \label{thm:poly} 
    Let $H$ be a graph and let $s$ be a positive integer. Then every graph with average degree at least $s^{500|V(H)|^2}$ contains either an induced subdivision of $H$ or a complete bipartite (not necessarily induced) subgraph $K_{s,s}$.
\end{theorem}

Let us say that a graph $G$ is $k$-connected if either $G$ is a complete graph with at least $k+1$ vertices, or if $G$ is a non-complete graph such that for every set $X$ of at most $k-1$ vertices in $G$, we have that $G \setminus X$ is connected. Alon, Kleitman, Thomassen, Saks, and Seymour \cite{alon} first showed that every graph with very large chromatic number contains a highly connected subgraph with large chromatic number and connectivity. The current best bounds for this result are due to Nguyen \cite{Nguyen}: 
\begin{theorem}[Nguyen \cite{Nguyen}]
\label{thm:narayanan}
	For each $k \in \mathbb{N}$, every graph $G$ with chromatic number at least $\left(3 + \frac{1}{16}\right)k$ contains a subgraph $H$ with connectivity and chromatic number each at least $k$.
\end{theorem}

Note that the subgraph in Theorem \ref{thm:narayanan} can be taken to be induced.

We need one more result, which requires defining the graphs we exclude. Let $a$ and $b$ be positive integers. Define $\cP_{a,b}$ to be the graph whose vertex-set can be partitioned into nonempty sets $\{v\},\{x,y\}, A, B$ with the following properties:
\begin{enumerate}[label=(\alph*)]
	\item $\{x,y\}$, $A$, and $B$ are independent sets, with $|A| = a$ and $|B| = b$;
	\item $v$ is complete to $\{x,y\}$;
	\item $\{x,y\}$ is complete to $A$;
	\item $A$ is complete to $B$;
	\item there are no edges other than the edges listed above.
\end{enumerate}
Note that  $\cP_{2,1}$ is $K_4^{++}$ (that is, $K_4$ with two disjoint edges each subdivided once).  In general, the graph $\cP_{a,b}$ is isomorphic to the complete bipartite graph $K_{a,b+2}$ together with an additional vertex $v$ adjacent to exactly two vertices $x$ and $y$ in the part of size $b+2$. 

A different way to describe $\cP_{a, b}$ is that it can be obtained from a four-vertex path by substituting independent sets of size $1, 2, a$, and $b$ for its vertices in order. In particular, a result of Alon, Pach, and Solymosi \cite{aps} implies that $\cP_{a,b}$-free graphs satisfy the Erd\H{o}s-Hajnal conjecture; in other words: 
\begin{theorem}[Alon, Pach, Solymosi \cite{aps}] \label{thm:aps}
There is a constant $c = c(a, b) > 0$ such that every $\cP_{a,b}$-free graph $G$ on $n$ vertices contains a clique or an independent set on $n^c$ vertices. 
\end{theorem}

\section{The Main Result}
We prove that the class of IS$\cP_{a,b}$-free graphs is $\chi$-bounded. Note that it is enough to prove the bound on the chromatic number when $a=b$, because   $\cP_{a,a}$ contains $\cP_{a',b'}$ as an induced subgraph  whenever $a',b' \leq a$. In our proof, we will assume that $a \ge 2$. Let $\cC_a =\ $IS$\cP_{a,a}$-free denote the class of graphs that do not contain an induced subdivision of $\cP_{a,a}$. Our main result is the following. 

\begin{theorem} \label{thm:main}
    Let $a \in \mathbb{N}$ with $a \geq 2$. Then $\cC_a$ is $\chi$-bounded. 
\end{theorem}

The proof of Theorem \ref{thm:main} starts with an idea (so-called ``templates'') first used by Kierstead and Penrice \cite{kp} and Kierstead and Zhu \cite{kzhu}. Theorem \ref{thm:kuhn-osthus} allows us to restrict our attention to graphs $G$ that contain a large complete bipartite graph, which we carefully grow into a complete multipartite graph (the ``template'') according to a set of rules. While the proofs in \cite{kp, kzhu} use this to partition the graph, we instead analyze how the remainder of $G$ attaches to our template, and find a small cutset. This allows us to leverage Theorem \ref{thm:narayanan} to bound the chromatic number. 

In Section \ref{section:attach}, we prove some initial results about how vertices attach to complete multipartite graphs, and in Section \ref{section:main}, we prove the main result. 

\section{Attaching to a complete multipartite graph} \label{section:attach}

Throughout the remainder of the paper, we fix $a \in \mathbb{N}$ with $a \geq 2$. Since Theorem \ref{thm:aps} can be used to improve bounds from applying Ramsey's theorem in our class of graphs, we define $\cR(p, q) = (p+q)^{1/c(a, a)}$ where $c(a,a)$ is defined as in Theorem \ref{thm:aps}. It follows that every graph in $\cC_a$ on at least $\cR(p,q)$ vertices contains a clique or an independent set of size at least $p+q$. 

\begin{lemma}
\label{lem:adjacency}
	Let $G \in \cC_a$ and let $A$ and $B$ be two independent sets in $G$ such that $A$ is complete to $B$. Let $v$ be a vertex in $G$ that is not in $A$ or $B$. If $v$ has at least $a$ non-neighbours in $A$, then $v$ is either 1-connected to $B$ or $(a-1)$-disconnected to $B$ (that is, $v$ cannot simultaneously have two neighbours and $a$ non-neighbours in $B$).
\end{lemma}

\begin{proof}
	For a contradiction, suppose that $v$ has $a$ non-neighbours $x'_1,\dots,x'_a \in A$, two neighbours $y_1,y_2 \in B$, and $a$ non-neighbours $y'_1, \dots, y'_a \in B$. Then $\{v\} \cup \{y_1,y_2\} \cup \{x'_1,\dots,x'_a\} \cup \{y'_1,\dots,y'_a\}$ would induce a copy of $\cP_{a,a}$, a contradiction.
\end{proof}
\medskip

We use $X = (X_1,X_2,\dots,X_r)$ to denote a complete $r$-partite graph with \emph{parts} $X_1,\dots,X_r$. In other words, $X_1, \dots, X_r$ are disjoint non-empty independent sets, and for $i \neq j$, every vertex in $X_i$ is adjacent to every vertex in $X_j$. 
The following observation follows from Lemma~\ref{lem:adjacency}. \medskip

\begin{lemma}
\label{lem:adjacency-type}
	Let $G \in \cC_a$ and let $X = (X_1, \dots, X_r)$ be a complete $r$-partite induced subgraph of $G$
    where each part has size at least $a+1$. Let $Z$ be the set of vertices in $G \sm X$ that are $(a-1)$-disconnected to every part of $X$. Let $v$ be a vertex of $G \sm (X \cup Z)$. Then the following holds. 
	
\begin{enumerate}[label=(\alph*)]	
	\item There exists a part $X_i$ such that $X_i$ contains at least $a$ non-neighbours of~$v$.
	
	\item For every part $X_j$ of $X$ with $j \neq i$, either $v$ is $1$-connected to $X_j$ or $v$ is $(a-1)$-disconnected to $X_j$.
	
	\item If there exists some $j \neq i$ such that $v$ is 1-connected to $X_j$, then $v$ is 1-connected to $X_i$.
\end{enumerate}
\end{lemma}

\begin{proof}
	Statement (a) follows from the fact that $Z$ contains all vertices that have at most $a-1$ neighbours in each part of $X$, and $v \notin Z$. 
	Statement (b) follows from (a) and Lemma~\ref{lem:adjacency}. 
	Statement (c) follows from the fact that $v$ has at most one neighbour in $X_j$ (meaning that $v$ has at least $a$ non-neighbours in $X_j$), hence by Lemma~\ref{lem:adjacency} applied to $X_j$ and $X_i$, it follows that $v$ does not have two neighbours in $X_i$, so $v$ is 1-connected to $X_i$.
\end{proof}
\medskip

\begin{lemma}
\label{lem:nearly-disconnected}	
	Let $s, w\in \mathbb{N}$. Let $G$ be a graph with $\omega(G) = w$, and let $X = (X_1, \dots, X_r)$ be a complete $r$-partite graph in $G$. Let $Z$ be the set of vertices in $G \sm X$ such that for every vertex $v \in Z$, $v$ is $a$-disconnected to $X_i$ for every $i \in \{1,2,\dots,r\}$.
	If for every $i \in \{1,2,\dots,r\}$, we have $|X_i| = s+a$, and $|Z| \geq \cR(w, s) (s+a+1)^{ra}$, then $X \cup Z$ contains an induced complete $(r+1)$-partite graph such that each of its parts contains at least $s$ vertices.
\end{lemma}	

\begin{proof}
By the Pigeonhole Principle, the set $Z$ contains a subset $J$ of $\cR(w, s)$ vertices such that for every pair of vertices $x$ and $y$ in $J$, $N(x) \cap X = N(y) \cap X$. To see this, suppose for a contradiction that such a set $J$ does not exist. Note that $X$ contains at most $(s+a+1)^{ra}$ different subsets that contain at most $a$ vertices in each part of $X$. By assumption, each subset can only be complete to a subset of $Z$ of size at most $\cR(w, s) - 1$. Hence the size of $Z$ would be at most $\cR(w, s) (s+a+1)^{ra}$, a contradiction. 

By Theorem \ref{thm:aps} applied to $J$, we know that $J$ contains an independent set $I$ of size at least $s$. Now form a graph $fX'$ by removing the non-neighbours of $I$ in every part of $X$. Note that at most $a$ vertices are removed from each part of $X$. So $I$ is complete to $X'$, and since $(s+a) - a =s$, every part of $X'$ has size at least $s$. Hence, $I \cup X'$ induces a complete $(r+1)$-partite graph, where each part has a size of at least $s$.
\end{proof}
\medskip

\section{Proof of the main theorem} \label{section:main}

\begin{theorem} \label{thm:maindetails}
Let $a \in \mathbb{N}$ with $a \geq 2$. Let $G' \in \cC_a$, and suppose that every induced subgraph $H$ of $G'$ with $\omega(H) < \omega(G')$  satisfies $\chi(H) \leq \tau$, for some $\tau \in \mathbb{N}$. Then we have 
$$\chi(G') < \left(3+\frac{1}{16}\right) b$$ 
where 
$$b = 1+\max\{d(\cP_{a,a}, \cR(\omega(G'), f)), (1+\tau)(a+1)\omega(G'), \cR(\omega(G'), f) (f+1)^{a\omega(G')}+ \omega(G')^2 (f^{a+1}+2)\},$$
and
$$f = a(\omega(G')+1) + 2.$$
\end{theorem}
\begin{proof}
	 Let $G', a, b, f$ be as in the statement of Theorem \ref{thm:maindetails}. Suppose that $\chi(G') \geq (3+\frac{1}{16})b$. By Theorem \ref{thm:narayanan}, it follows that $G'$ contains an induced subgraph $G$ such that $\chi(G) \geq b$ and $G$ is $b$-connected. Define $s = \cR(\omega(G), f)$. If $G$ does not contain $K_{s,s}$ as a subgraph, by Theorem~\ref{thm:kuhn-osthus}, every induced subgraph of $G$ has an average degree less than $d = d(\cP_{a,a}, s)$. Thus, $G$ is $(d-1)$-degenerate, and hence $d$-colourable. This contradicts that $\chi(G) \geq b \geq d+1$; and so Theorem \ref{thm:maindetails} holds.
    
    From now on, we may assume that $G$ contains a copy $Y$ of $K_{s,s}$ as a subgraph. Since $s = \cR(\omega(G), f)$, and since the graph induced in $G$ by each side of the bipartition of $Y$ has clique number at most $\omega(G) - 1$, it follows that $G[V(Y)]$ contains $K_{f,f}$ as an induced subgraph. Consider the largest $r$ such that $G$ contains a complete $r$-partite graph where each part has size exactly $f - a(r-2)$; call this complete $r$-partite graph $X = (X_1, \dots, X_r)$. Note that $r \leq \omega(G)$. Furthermore, since $G$ contains an induced copy of $K_{f,f}$, it follows that $r \geq 2.$

    Let $Z$ be the set of vertices that are $a$-disconnected to $X_i$ for all $i \in \{1,\dots,r\}$ (that is, every vertex of $Z$ has at most $a$ non-neighbours in every part of $X$). If $$|Z| \geq \cR(\omega(G), f - a(r-1)) (f - a(r-2)+1)^{ra},$$ then by Lemma~\ref{lem:nearly-disconnected}, it follows that $G$ contains a complete $(r+1)$-partite graph where each part is of size at least $f - a(r-1)$, contradicting the maximality of $X$. So we may assume that $|Z| < \cR(\omega(G), f - a(r-1)) (f - a(r-2)+1)^{ra} \leq \cR(\omega(G), f)(f+1)^{a\omega(G)}$.
The following claim follows from Lemma~\ref{lem:adjacency-type}.
\medskip

\begin{claim}
\label{cl:1}
	For every vertex $v \in G \sm (X \cup Z)$, the following hold.
	\begin{enumerate}[label=(\alph*)]
		\item If there exist at least two parts of $X$ that each contain at least $a$ non-neighbours of $v$, then $v$ is 1-connected to each such part, and $(a-1)$-disconnected to every other part.
		
		\item If there exists a unique part $X_i$ of $X$ that contains at least $a$ non-neighbours of $v$, then $v$ is $(a-1)$-disconnected to every other part, and $v$ has at least $a+1$ non-neighbours in $X_i$. 
\end{enumerate}	 
\end{claim}
Using Claim \ref{cl:1}, we partition $V(G \setminus (X \cup Z))$ into three sets: 
\begin{itemize}
\item $C$ is the set of vertices $v$ in $G \sm (X \cup Z)$ for which there is a unique part of $X$ containing at least $a$ non-neighbours of $v$; 
\item $A$ is the set of vertices in $G \sm (X \cup Z)$ that are 1-connected to every part of $X$; and
\item $M$ is the set of vertices in $G \sm (X \cup Z)$ that are $(a-1)$-disconnected to at least one part of $X$ and are 1-connected to at least two parts of $X$. 
\end{itemize}

We will study two cases, namely when $A = \emptyset$ and when $A \neq \emptyset$.
\medskip

\begin{claim} \label{claim:2}
	If $A = \emptyset$, then $G$ can be coloured using $(1+ \tau) (a+1) \omega(G)$ colours. 
\end{claim}
\medskip

\noindent \textit{Proof of Claim \ref{claim:2}.}
Consider a set $S$ containing exactly $a+1$ vertices of each part of $X$. We have $|S| \leq (a+1)\omega(G)$ because $r \leq \omega(G)$. We show that $S$ is a dominating set of $G$.
First, every vertex of $X \sm S$ has a neighbour in $S$.
Next, every vertex of $Z$ has a neighbour in $S$ (because every vertex in $Z$ has at most $a$ non-neighbours in every part $X_i$ and $|X_i \cap S| = a+1$).  
Last, every vertex $v$ in $C \cup M$ has a neighbour in $S$ because there exists a part $X_i$ of $X$ to which $v$ is $(a-1)$-disconnected.

Now for every vertex $v \in S$, note that the size of the largest clique in $N_{G-S}(v)$ is at most $\omega(G)-1$. Hence, by assumption, $N_{G-S}(v)$ can be coloured using $\tau$ colours. Since $S$ contains at most $(a+1)\omega(G)$ vertices and $S$ is a dominating set, it follows that $\chi(G \setminus S) \leq \tau (a+1) \omega(G)$. Giving each vertex of $S$ an additional unique colour proves the claim.
\medskip

From now on, we assume that $A \neq \emptyset$. Our goal will be to show that each connected component of $A$ can be separated from $X$ by deleting a small set of vertices. 
\medskip

By the definition of $C$ and by Claim \ref{cl:1}, for every vertex $v$ in $C$, there exists a unique part that contains at least $a+1$ non-neighbours of $v$; call this the \emph{conflict part} of $v$. We partition the vertices in $C$ into at most $r$ sets, based on their conflict part: For every $i \in \{1,2,\dots,r\}$, denote by $C_i$ the set of vertices in $C$ that have conflict part $X_i$.

For $i, j \in \{1, \dots, r\}$ with $i \neq j$ and with $i = 1$ or $j = 1$, we let $M_{ij}$ be the subset of $M$ containing all vertices $v \in M$ such that: 
\begin{itemize}
    \item $i$ is the minimum index such that $v$ is $(a-1)$-disconnected to $X_i$; 
    and 
    \item $j$ is the minimum index such that $v$ is 1-connected to $X_j$. 
\end{itemize}
By the definition of $M$ and Claim \ref{cl:1}, every vertex of $M$ is either $(a-1)$-disconnected or 1-connected to every part of $X$. It follows that $M_{12}, M_{13}, \dots, M_{1r}, M_{21}, M_{31}, \dots, M_{r1}$ is a partition of $M$ into at most $2r-2$ parts.

We will show that every connected component of $A$ has a bounded number of neighbours in $Z$, in $X_i$ for every $i \in \{1,\dots,r\}$, in $C_i$ for every $i \in \{1,\dots,r\}$, and in $M_{ij}$ for every possible pair $i,j$. Let $Q$ be a connected component  of $A$. Using our bound on $|Z|$, we immediately obtain: 
\medskip

\begin{claim}
\label{cl:Z}
	$Q$ has at most $\cR(\omega(G), f)(f+1)^{a\omega(G)} - 1$ neighbours in $Z$.
\end{claim}
\medskip

Next, we show: 
\begin{claim}
\label{cl:X}
For every $i \in \{1,\dots,r\}$, $Q$ has at most one neighbour in $X_i$.
\end{claim}
\medskip

\noindent \textit{Proof of Claim \ref{cl:X}.} 
For the sake of contradiction, suppose that there exists some $i$ such that $X_i$ contains at least $2$ neighbours of $Q$. 
Because $Q$ is connected, it follows that there is a path in $Q$ with at least two neighbours in $X_i$. Let $P$ be a shortest path in $Q$ such that there exists an $i' \in \{1, \dots, r\}$ with the property that $P$ has at least two neighbours in $X_{i'}$. By symmetry, we may assume $i' = i$. Let $x', y'$ be the ends of $P$. Let $x$ be the neighbour of $x'$ in $X_i$, and let $y$ be the neighbour of $y'$ in $X_i$. 

By the minimality of $P$, it follows that $xx'Py'y$ is induced, and both $P \setminus \{x'\}$ and $P \setminus \{y'\}$ each have at most one neighbour in each part of $X$.  So $|N(P) \cap X_j| \leq 2$ for every $1 \leq j \leq r$. 

Consider a part $X_k$ for some $k \neq i$ (using that $r \geq 2$), and let $X'_k$ be obtained from $X_k$ by deleting the neighbours of $P$. Let $X'_i = X_i \sm \{x,y\}$. Since each part of $X$ has size at least $f - (a-1)(r-2) \geq f - (a-1) (\omega(G)-2) \geq a+2$, it follows that $|X'_i|, |X'_k| \geq a$. Moreover, $P$ is anticomplete to $X'_i$ and $X'_k$, and $X'_i$ and $X'_k$ are complete to each other. Hence, $X'_i \cup X'_k \cup \{x,y\} \cup P$ contains an induced subdivision of $\cP_{a,a}$, a contradiction. This proves Claim \ref{cl:X}.
\medskip

\begin{claim}
\label{cl:M}
	For every possible pair $i,j$, $Q$ has at most $\omega(G)$ neighbours in $M_{ij}$.
\end{claim}
\medskip

\noindent \textit{Proof of Claim \ref{cl:M}.} 
For the sake of contradiction, suppose without loss of generality that $M_{12}$ contains at least $\omega(G)+1$ neighbours of $Q$,  where by definition, every vertex in $M_{12}$ is $(a-1)$-disconnected to $X_1$ and is 1-connected to $X_2$. Since $|N(Q) \cap M_{12}| \geq \omega(G)+1$, it follows that $N(Q) \cap M_{12}$ contains two distinct non-adjacent vertices $x, y$. Let $P$ be a shortest path from $x$ to $y$ with interior in $Q$. 

By Claim~\ref{cl:X}, $Q$ has at most one neighbour in each of $X_1$ and $X_2$. Let $X_1' = (X_1 \cap N(x) \cap N(y)) \setminus N(P \setminus \{x, y\})$. Since $x, y$ each have at most $a-1$ non-neighbours in $X_1$, and since $Q$ has at most one neighbour in $X_1$, it follows that $|X_1'| \geq |X_1| - 2a + 1 \geq a$. Let $X_2' = X_2 \setminus N(P)$. Since each of $Q, x, y$ have at most one neighbour in $X_2$, it follows that $|X_2'| \geq |X_2| - 3 \geq a$. Now $X'_2 \cup X'_1 \cup \{x,y\} \cup P$ contains an induced subdivision of $\cP_{a,a}$, a contradiction.
This proves Claim \ref{cl:M}.
\bigskip

\begin{claim}
\label{cl:C}
	 For every $i \in \{1, \dots, r\}$, $Q$ has at most $\omega(G) f^{a+1}$ neighbours in $C_i$.
\end{claim} 
\medskip

\noindent \textit{Proof of Claim \ref{cl:C}.} For the sake of contradiction suppose without loss of generality that $C_1$ contains more than $\omega(G) f^{a+1}$ neighbours of $Q$; let $R = N(Q) \cap C_1$. For each vertex $v \in R$, let $S_v$ be a set of exactly $a+1$ non-neighbours of $v$ in $X_1$. Since $|X_1| = f - a(r-2) \leq f$, it follows that there are at most $f^{a+1}$ possible choices of $S_v$. By the Pigeonhole Principle, it follows that there is a set $R' \subseteq R$ and a set $S \subseteq X_1$ with $|S| = a+1$ such that $S_v = S$ for all $v \in R'$, and $|R'| > \omega(G)$. Consequently, $R'$ contains two distinct non-adjacent vertices, say $x$ and $y$. 

Since $Q$ is connected, 
there exists a shortest (and therefore induced) path $P$ with interior in $Q$ connecting $x$ and $y$. Let $X'_2$ be obtained from $X_2$ by removing at most $2(a-1)+1 = 2a-1$ vertices including all non-neighbours of $x$ and $y$, and the unique neighbour of $P \setminus \{x, y\}$ (if any). Then $\{x,y\}$ is complete to $X'_2$, and $P \setminus \{x, y\}$ is anticomplete to $X'_2$. Note that $|X'_2| \geq 3a+2-(2a-1) \geq a$. Let $S' = S \setminus N(Q)$. It follows that $|S'| \geq a$ by Claim \ref{cl:X}. Now $S' \cup X'_2 \cup \{x,y\} \cup P$ contains an induced subdivision of $\cP_{a,a}$, a contradiction. 
This proves Claim \ref{cl:C}.

\bigskip

Now we compute the number of neighbours of $Q$ in $G \sm A$.
By Claims~\ref{cl:X}-\ref{cl:C}, it follows that:
\begin{align*} 
|N(Q)| &< \cR(\omega(G), f)(f+1)^{a\omega(G)} + \omega(G) +  \omega(G)^2 f^{a+1} + (2\omega(G)-2)\omega(G)\\
&\leq \cR(\omega(G), f)(f+1)^{a\omega(G)}  +  \omega(G)^2 (f^{a+1} + 2) \leq b. 
\end{align*}
In particular, $Q$ has at most $r$ neighbours in $X$ by Claim \ref{cl:X}, and so $N(Q)$ is a cutset in $G$ separating $Q$ (which is non-empty) from $X \setminus N(Q)$ (which is also non-empty). Since $|N(Q)| < b$, this contradicts the fact that $G$ is $b$-connected, a contradiction. This completes the proof.
\end{proof}

We note that the overall bound is not polynomial. While Theorems \ref{thm:aps} and \ref{thm:poly} yield polynomial bounds for $d(\cP_{a,a}, \cdot)$ and $\cR(\cdot, \cdot)$, two places remain that do not lend themselves to a polynomial bound. The first is the usage of $\tau$ (and thus implicit induction on $\omega(G)$), which shows that our current proof gives a bound of at least $\omega(G)^{\omega(G)}$. The second is the use of Lemma \ref{lem:nearly-disconnected}, which again is in the $\omega^{\mathcal{O(\omega)}}$ range. However, it seems unlikely that this is the best possible bound, so we have not optimized this further. 

To compute the bound given by our result, we start by computing $$\cR(\omega(G), f) \leq (\omega(G)+f)^{1/c(a,a)} \leq ((a+1)(\omega(G)+1))^{1/c(a,a)}$$
using Theorem \ref{thm:aps}. Next, using Theorem \ref{thm:poly}, we obtain that 
$$d(\cP_{a,a}, \cR(\omega(G), f)) \leq ((a+1)(\omega(G)+1))^{\frac{500(2a+3)^2}{c(a,a)}}.$$
Therefore, 
\begin{align*}
    \chi(G) \leq 4\left(((a+1)(\omega(G)+1))^{\frac{500(2a+3)^2}{c(a,a)}} + (1+\tau)(a+1)\omega(G) + ((a+1)(\omega(G)+1))^{\frac{a \omega(G)}{c(a,a)}} \right)
\end{align*}
using that $\omega(G)^2(f^{a+1}+2) \leq (\omega(G)+f)^{2a+2} \leq (\omega(G)+f)^{1/c(a,a)}.$
From here, a straight-forward induction shows that 
$$\chi(G) \leq \left(((a+1)(\omega(G)+1))^{\frac{500(2a+3)^2}{c(a,a)}} + ((a+1)(\omega(G)+1))^{\frac{a \omega(G)}{c(a,a)}} \right)(8(a+1)(\omega(G)))^{\omega(G)}.$$

\bibliographystyle{plain}

\end{document}